\documentclass{article}
\usepackage{amsmath,amsthm,amsfonts,amssymb,amscd,enumerate}
\RequirePackage{graphics}

\theoremstyle{plain}
\newtheorem{theorem}{Theorem}[section]
\newtheorem{proposition}[theorem]{Proposition}

\newtheorem{corollary}[theorem]{Corollary}
\newtheorem{lemma}[theorem]{Lemma}
\newtheorem*{decomp}{The Decomposition Theorem}
\newtheorem*{theoremA}{The Main Theorem}
\newtheorem*{Theorem}{Theorem}

\theoremstyle{definition}

\newtheorem{definition}[theorem]{Definition}

\newtheorem{example}[theorem]{Example}

\newtheorem{remark}[theorem]{Remark}

\newtheorem*{remarks*}{Remarks}
\newcommand{\ut}{{>T}}

\newcommand{\cb}{\mathcal {B}}
\newcommand{\cac}{\mathcal {C}}

\newcommand{\cf}{\mathcal {F}}

\newcommand{\ch}{\mathcal {H}}
\newcommand{\ci}{\mathcal {I}}

\newcommand{\cn}{\mathcal {N}}

\newcommand{\cp}{\mathcal {P}}

\newcommand{\cs}{\mathcal {S}}
\newcommand{\ct}{\mathcal {T}}
\newcommand{\cu}{\mathcal {U}}

\newcommand{\sphere}{{\mathbb S}}

\newcommand{\rr}{{\mathbf R}}

\newcommand{\zz}{{\mathbf Z}}

\newcommand{\minus}{{-1}}
\newcommand{\fin}{{\mathrm {fin}}}

\newcommand{\hA}{{\widehat A}}

\newcommand{\bs}{{\mathbf s}}

\newcommand{\wt}{\widetilde}

\newcommand{\ol}{\overline}

\newcommand{\ha}{{\hat{A}}}

\newcommand{\ga}{\alpha}

\newcommand{\gd}{\delta}
\newcommand{\geps}{\varepsilon}
\newcommand{\gf}{\varphi}
\newcommand{\gi}{\iota}
\newcommand{\gr}{\rho}
\newcommand{\gs}{\sigma}
\newcommand{\gt}{\tau}

\newcommand{\gG}{\Gamma}
\newcommand{\gD}{\Delta}
\newcommand{\gdo}{\Delta^{op}}
\newcommand{\pgdo}{\partial\Delta^{op}}

\newcommand{\gU}{\Upsilon}

\newcommand{\In}{\operatorname{In}}

\newcommand{\Aut}{\operatorname{Aut}}
\newcommand{\Card}{\operatorname{Card}}

\newcommand{\Lk}{\operatorname{Lk}}
\newcommand{\plk}{\operatorname{PLk}}

\newcommand{\cohd}{\operatorname{cd}}
\newcommand{\vcd}{\operatorname{vcd}}

\newcommand{\cat}{\operatorname{CAT}}
\newcommand{\flag}{\operatorname{Flag}}

\newcommand\mapright[1]{\smash{\mathop{\longrightarrow}\limits^{#1}}}

\newenvironment{enumeratei}{\begin{enumerate}[\upshape (i)]}%
        {\end{enumerate}}
\newenvironment{enumerate1}{\begin{enumerate}[\upshape 1)]}%
        {\end{enumerate}}

\newenvironment{enumeratea}{\begin{enumerate}[\upshape 
(a)]}{\end{enumerate}}
\newenvironment{enumeratea'}{\begin{enumerate}[\upshape 
(a)$'$]}{\end{enumerate}}

\newcommand{\comment}[1]{}

\DeclareMathOperator{\V}{\bf V}

\providecommand{\bysame}{\makebox[3em]{\hrulefill}\thinspace,\ }

\numberwithin{equation}{section}
\begin{document}
\title{Compactly supported cohomology of buildings} 
\author{M.W. Davis\thanks{The first author was partially supported by NSF 
grant DMS 0706259.}
\and
J. Dymara\thanks{The second author was partially supported by KBN grant 
N201 012 32/0718}  
\and
T. Januszkiewicz\thanks{The third 
author also was partially supported by NSF grant DMS 0706259.}
\and
J. Meier\thanks{The fourth author was partially supported by a Richard King Mellon research grant.}
\and
B. Okun}

\date{}
\maketitle
\begin{abstract}
We compute the compactly supported cohomology of the standard realization of any locally finite building.

\smallskip

\noindent
\textbf{AMS classification numbers}. Primary: 20F65 \\
Secondary: 20E42, 20F55, 20J06, 57M07.
\smallskip

\noindent
\textbf{Keywords}:  Building, cohomology of groups, Coxeter group.
\end{abstract}

\section*{Introduction}
A building consists of a set $\Phi$ (the elements  of which are called ``chambers'') together with a family of equivalence relations (``adjacency relations'') on $\Phi$ indexed by a set $S$ and a ``$W$-valued distance function,'' $\Phi\times \Phi \to W$, where $W$ is a Coxeter group with fundamental set of generators $S$.  So, associated to any building  there is a Coxeter system $(W,S)$, its \emph{type}.   


There is a construction which associates a topological space to $\Phi$.  This construction admits some freedom of choice.  The idea is to choose a space $X$ as a ``model chamber'' and then glue together copies of it, one for each element of $\Phi$.  To do this, it  is first necessary to choose a family of closed subspaces $\{X_s\}_{s\in S}$ so that copies of $X$ corresponding to $s$-adjacent chambers are glued together along $X_s$. (We call such a family, $\{X_s\}_{s\in S}$, a ``mirror structure'' on $X$.)  Let $\cu(\Phi,X)$  denote the topological realization of $\Phi$ where each chamber is  realized by a copy of the model chamber $X$. (Details are given in Section~\ref{s:geom}.)

Classically, interest  has centered on buildings of spherical or affine type, meaning that $W$ is a spherical or Euclidean reflection group, respectively.  For example, each algebraic group over a local field has a corresponding affine building.  However, we are mainly interested in buildings which are not classical in that their associated Coxeter systems are neither spherical nor affine.  This is a large class of spaces, many of which have a great deal of symmetry.  For example, such buildings arise in the theory of Kac-Moody groups (e.g., see \cite{cr, remy,rr}).  Also,  nonclassical buildings associated to arbitrary right-angled Coxeter groups have been a subject of recent interest in geometric group theory (e.g., see \cite{dbuild,js,  thomas}).

Two choices for a model chamber $X$ stand out.  The first is $X=\gD$, a simplex of dimension $\Card(S)-1$, with its codimension one faces indexed by  $S$.  This was Tits' original choice (cf.~\cite{ab}).  We call $\cu(\Phi, \gD)$ the ``classical realization'' of $\Phi$.  The other choice for $X$ is the ``Davis chamber'' $K$, defined as the geometric realization of the poset $\cs$ of spherical subsets of $S$ (see \cite[Chapters 7, 18]{dbook}).  $\cu(\Phi, K)$ is the ``standard realization'' of $\Phi$.  Both realizations are contractible.  The standard realization is  important in geometric group theory, the reason being that in this field one is interested in discrete group actions which are both proper and cocompact and these conditions are more likely to hold for the action of a  group of automorphisms on the standard realization than on the classical realization.  (If $\Phi$ has finite thickness, $\cu(\Phi,K)$ is locally finite, while $\cu(\Phi,\gD)$ need not be.)  

If a discrete group $\gG$ acts properly and cocompactly on a locally finite, contractible CW complex $Y$, then the compactly supported cohomology  of $Y$ is isomorphic  to the cohomology of $\gG$ with $\zz \gG$ coefficients.  In particular, it determines the virtual cohomological dimension of $\gG$, as well as, the number of ends of $\gG$, and it determines if $\gG$ is a duality group.  (For more information, see \cite[Part IV]{g}.) 
However, as we will explain, even if one is only interested in cohomological computations in the case of the standard realization of $\Phi$, it is necessary to carry out similar computations for various other realizations, in particular, for the classical realization.

In the classical case of an (irreducible) affine building, the two notions of model chamber agree:  $\gD=K$.  So, in the affine case the study of the  cohomology of cocompact lattices in $\Aut (\Phi)$ is closely tied to the study of the cohomological properties of $\cu(\Phi,\gD)$.  For example, in \cite{bs} Borel and Serre calculated the compactly supported cohomology, $H^*_c(\cu(\Phi, \gD))$, for any (irreducible) affine building and then used this calculation to derive information about the cohomology of ``$S$-arithmetic'' subgroups. The calculation of \cite{bs} was that $H^*_c(\cu(\Phi, \gD))$ is concentrated in the top degree ($=\dim\gD$) and is free abelian in that degree. 

Our main result, Corollary~\ref{c:main},  is a calculation of $H^*_c(\cu(\Phi, K))$, generalizing the theorem of Borel-Serre. In the case where $\Phi=W$, this was done in \cite{d98,ddjo2,dm}.  For a general (thick) building, in the case where $(W,S)$ is right-angled, it was done in \cite[Theorem 6.6]{ddjo2}.  It was claimed in full generality in \cite{dm}; however, there is a mistake in the proof (see \cite{dm-erratum}).

In order to write the formula, we need  more notation.  Let $A$ be the free abelian group of finitely supported, $\zz$-valued functions on $\Phi$.  For each subset $T\subseteq S$, let $A^T$ denote the subgroup of all functions $f\in A$ which are constant on each residue of type $T$.  (A ``residue'' of type $T$ is a certain kind of subset of $\Phi$; in the case of the building $W$, a residue of type $T$ is a left coset of $W_T$, the subgroup of $W$ generated by $T$.)  N.B. The empty set, $\emptyset$, is a spherical subset and a residue of type $\emptyset$ is just a single chamber; hence, $A^\emptyset=A$.  
If $U\supset T$, then $A^U\subset A^T$.  Let $A^\ut \subset A^T$ denote the $\zz$-submodule, $\sum_{U\supset T} A^U$.  (Throughout this paper we will use the convention that $\subseteq$ denotes containment and $\subset$ will be reserved for strict containment.  Also, we will use the symbol $\sum$ for an internal sum of modules and $\bigoplus$ to mean either an external direct sum or an internal sum which we have proved is direct.) 

We shall show in Section~\ref{s:decomp} that $A^\ut$ is a direct summand of $A^T$.  Let $\ha^T$ be a complementary summand.  As in \cite{ddjo2} the main computation is a consequence the following Decomposition Theorem (proved as Theorem~\ref{t:decomp} in Section~\ref{s:decomp}).

\begin{decomp}
\[
A=\bigoplus_{U\in \cs} \ha^U\quad\text{and, in fact, for any $T\in\cs$,}\quad A^T=\bigoplus_{U\supseteq T} \ha^U.
\]
\end{decomp}

The point is that this theorem provides a decomposition of a coefficient system which can be used to calculate the compactly supported cohomology of any of the various realizations of $\Phi$.  The calculation in which we are most interested is the following (proved as Corollary~\ref{c:main} in Section~\ref{s:cohomology}).

\begin{theoremA}
Suppose $\Phi$ is a building of finite thickness and type $(W,S)$. Let $K$ be the geometric realization of the poset $\cs$ of spherical subsets of $S$.  Then
\[
H^*_c(\cu(\Phi,K))\cong \bigoplus_{T\in \cs} H^*(K,K^{S-T})\otimes \ha^T.
\]
\end{theoremA}
\noindent
(For any subset $U$ of $S$, $K^U$ denotes the union of the $K_s$, with $s\in U$.)

The above theorem applies to all buildings.  A general building $\Phi$ will not be highly symmetric, in that its automorphism group, $\Aut (\Phi)$, can have infinitely many orbits of chambers.  However, if $\Phi/\Aut(\Phi)$ is finite and if $\gG$ is a torsion-free cocompact lattice in $\Aut (\Phi)$, then the Main Theorem  implies that the cohomological dimension of $\gG$ is equal to the virtual cohomological dimension of the corresponding Coxeter group.  Moreover, this dimension is 
\[
\cohd(\gG)=\vcd(W)=\max\{k\mid H^k(K,K^{S-T})\neq 0, \text{for some $T\in \cs$}\}
\]
(cf.~Corollary~\ref{c:dim}).  As another example, such a torsion-free cocompact lattice is an $n$-dimensional duality group if and only if for each $T\in \cs$, $H^*(K,K^{S-T})$ is free abelian and concentrated in degree $n$ (cf.~Corollary~\ref{c:duality}).

The central objective of \cite{ddjo2} was to calculate $H^*_c(\cu(W,K))$ as a $W$-module.   In that paper we showed there is a filtration of $H^*_c(\cu(W,K))$ by $W$-submodules so that the associated graded terms look like the terms on the right hand side of the formula in the Main Theorem.  Similarly, one can ask about the $G$-module structure of $H^*_c(\cu(\Phi,K))$ for any subgroup $G\subseteq \Aut (\Phi)$.
The methods of \cite{ddjo2} are well adapted to the present paper.  In particular, for each $T\in \cs$, the free abelian group $A^T$ is a $G$-module, as is its quotient $D^T:=A^T/A^\ut$.  So, as in \cite{ddjo2}, there is a filtration of $H^*_c(\cu(\Phi,K))$ by $G$-submodules and we get the following (proved as Theorem~\ref{t:Gmodule}).  

\begin{Theorem}
Suppose $G$ is a group of automorphisms of $\Phi$.  There is a filtration of $H^*_c(\cu(\Phi,K))$ by right $G$-submodules with associated graded term in filtration degree $p$:
\[ 
\bigoplus _{\substack{T\in \cs\\|T|=p}} H^*(K,K^{S-T})\otimes D^T.
\]
\end{Theorem}

As we mentioned earlier, although the Main Theorem is the result of importance in geometric group theory, it is no harder to do similar calculations when the chambers are modeled on an arbitrary $X$ (or at least on $X^f$, the complement of the faces of $X$ which have infinite stabilizers in $W$).  In fact, as we explain below, the proof of the Decomposition Theorem depends on first establishing a version of the Main Theorem for $\cu(\Phi,\gD^f)$.  So, the Main Theorem ultimately depends on first proving a version of it in the case of the classical realization. This version (proved as Theorem~\ref{t:SI}) is the following. 

\begin{Theorem}
When $W$ infinite,  $H^*_c(\cu(\Phi, \gD^f))$ is free abelian and is concentrated in the top degree $n$ ($=\dim \gD$).  
\end{Theorem}

This result is obvious when $\Phi=W$, for then $\cu(W,\gD^f)$ is homeomorphic to Euclidean space $\rr^n$ (see 
Section~\ref{s:geom}).  We prove it for a general $\Phi$  by showing the following (Theorem~\ref{t:cat}). 

\begin{Theorem}
$\cu(\Phi,\gD^f)$ admits a $\cat(0)$ metric (extending Moussong's $\cat(0)$ metric on the standard realization).
\end{Theorem} 

The existence of this $\cat(0)$ metric on $\cu(\Phi,\gD^f)$ is of independent interest. 
(See \cite{moussong, dbuild} or \cite{dbook} for a description of Moussong's metric on $\cu(\Phi, K)$.)

To finish the calculation of $H^*_c(\cu(\Phi, \gD^f))$ we  invoke a result of \cite{bbm} which asserts that the compactly supported cohomology of such a $\cat(0)$ space is concentrated in the top degree provided the cohomology of each ``punctured link'' vanishes except
in the top degree.  These links are spherical buildings and the vanishing of the cohomology groups, in degrees below the top, of their punctured versions is a result of \cite{dymo} (and independently, \cite{someone}).

As we have said, special cases of the Decomposition Theorem were proved in  \cite{ddjo2}.  The method of \cite{ddjo2} was simply to find a basis for $A$ adapted to its decomposition into the $\ha^T$.  In the general case, finding an explicit description of such a basis seems problematic.  We use instead an idea coming from an analogy with the argument of \cite{ddjo}.  In that paper we proved $L^2$ versions of the Decomposition Theorem and of the Main Theorem.  The proof of the $L^2$ version of the Decomposition Theorem was homological:  the key step was to show that certain ``weighted $L^2$-homology'' groups of certain auxiliary spaces associated to $W$ vanished except in the bottom  degree.  We then applied a certain duality (not applicable here) to deduce the Decomposition Theorem in the cases of actual interest.  Analogously, in this paper we prove the Decomposition Theorem by establishing the vanishing, except in the top degree, of the cohomology of certain auxiliary spaces. The most important of these auxiliary spaces is $\cu(\Phi,\gD^f)$ and we indicated in the previous paragraphs how we prove the result in that case. 

\section{Coxeter groups and buildings}\label{s:buildings} 
A \emph{chamber system over a set $S$} is a set $\Phi$ of \emph{chambers} together with a family of equivalence relations on $\Phi$ indexed by $S$.  Two chambers are \emph{$s$-equivalent} if they are related via the equivalence relation with index $s$; they are \emph{$s$-adjacent} if they are $s$-equivalent and not equal.
 A \emph{gallery} in $\Phi$ is a finite sequence of 
chambers $(\gf_0, \dots ,\gf_k)$ such that $\gf_{j-1}$ is adjacent to $\gf_j, 1 \le j \le k$.  The  \emph{type} of this gallery is the word $\bs
=(s_1,  \dots, s_k)$ where $\gf_{j-1}$ is $s_j$-adjacent to $\gf_j$.  If each $s_j$ 
belongs to a given subset $T$ of $S$, then the gallery is a \emph{$T$-gallery}. 
A chamber system is \emph{connected} (resp.,  \emph{$T$-connected}) if any 
two chambers can be joined by a gallery (resp., a $T$-gallery).  
The $T$-connected components of a chamber system $\Phi$ are its 
\emph{residues} of  \emph{type $T$}.  
 
A \emph{Coxeter matrix over a set} $S$ is an $S\times S$ symmetric matrix $M=(m_{st})$ with each diagonal entry $=1
$ and each off-diagonal entry an integer $\ge 2$ or the symbol $\infty$.  The matrix $M$ defines a presentation of a group $W$ as follows:  the set of generators is $S$ and the relations have the form $(st)^{m_{st}}$ where $(s,t)$ ranges over all pairs in $S\times S$ such that $m_{st}\neq \infty$.  The pair $(W,S)$ is a \emph{Coxeter system} (cf. ~\cite{bourbaki,dbook}).  Given  $T\subseteq S$, $W_T$ denotes the subgroup generated by $T$; it is called a \emph{special subgroup}.  $(W_T,T)$ is itself a Coxeter system.  The subset $T$ is \emph{spherical} if $W_T$ is finite.  

\begin{definition}
The poset of spherical subsets of $S$ (partially ordered by inclusion) is denoted $\cs$.
\end{definition}

Also,  $\cp$ ($=\cp(S)$) will denote the poset of all proper subsets of $S$.   (In Section~\ref{s:cat0} the poset $\cp-\cs$ plays a role.)

Suppose $(W,S)$ is a Coxeter system and $M=(m_{st})$ is its Coxeter matrix.  
Following  \cite{ronan} (or \cite{ddjo2}), 
a \emph{building of type $(W,S)$} (or of \emph{type $M$}) is  a chamber system $\Phi$ over $S$ such that 
\begin{enumeratei}
\item
for all $s\in S$, each $s$-equivalence class contains at least two chambers, and 
\item
there exists a \emph{$W$-valued  distance function}  $\gd: \Phi \times 
\Phi \to W$.  (This means that given  a reduced word  $\bs$ for an element $w\in W$, chambers $\gf$ and $\gf'$ can be joined by a gallery 
of type $\bs$ from $\gf$ to $\gf'$ if and only if $\gd (\gf, \gf') = w$.) 
\end{enumeratei}
\begin{example}\label{ex:thin}
The group $W$ itself has the structure of a building: the $s$-equivalence classes are the left cosets of $W_{\{s\}}$ and the $W$-valued distance, $\gd:W\times W\to W$, is defined by $\gd(v,w)=v^\minus w$.
\end{example}

A residue of type $T$ is  a building; its type is $(W_T,T)$.  
A building of type $(W,S)$ is \emph{spherical} if $W$ is finite.  A building has \emph{finite thickness} if each $s$-equivalence class is finite, for each $s\in S$.  (This implies all spherical residues are finite.)  \emph{Henceforth, all buildings will be assumed to have finite thickness}.

\section{Geometric realizations of Coxeter groups and buildings}\label{s:geom}
A \emph{mirror structure over a set} $S$ on a space $X$ is a family of subspaces $(X_s)_{s\in S}$ indexed by $S$.  Given a mirror structure on $X$, a subspace $Y\subseteq X$ inherits a mirror structure by $Y_s:=Y\cap X_s$.  If $X$ is a CW complex and each $X_s$ is a subcomplex, then $X$ is a \emph{mirrored CW complex}.  For each nonempty subset $T\subseteq S$, define subspaces $X_T$ and $X^T$ by
\begin{equation}\label{e:T}
X_T:=\bigcap_{s\in T} X_s\quad\text{and}\quad X^T:=\bigcup_{s\in T} X_s. 
\end{equation}
Put $X_\emptyset:=X$ and $X^\emptyset:=\emptyset$.  
Given a cell $c$ of (a CW complex) $X$ or a point $x \in X$, put
\begin{align*}
S(c)&:=\{s\in S\mid c\subseteq X_s\},\\
S(x)&:=\{s\in S\mid x\in X_s\}.
\end{align*}

Suppose now that $S$ is the set of generators for a Coxeter system $(W,S)$.  
Let $\gU(X)$ denote the union of the nonspherical faces of $X$ and $X^f$ its complement in $X$, i.e.,
\begin{equation}\label{e:xf}
\gU(X):=\bigcup_{T\notin \cs} X_T\quad \text{and}\quad X^f:=X-\gU(X).
\end{equation}
The mirror structure is \emph{$W$-finite} if $\gU(X)=\emptyset$.

Given a building $\Phi$ of type $(W,S)$ and a mirrored space $X$ over $S$,
define an equivalence relation  $\sim$ on $\Phi \times X$ by $(\gf, x) \sim (\gf', x')$ if and 
only if $x = x'$ and $\gd (\gf, \gf') \in W_{S(x)}$ (i.e., $\gf$ and $\gf'$ belong to the same $S(x)$-residue).  The $X$-\emph{realization} 
of $\Phi$, denoted  $\cu(\Phi,X)$, is defined by
	\begin{equation}\label{e: defUPhi}
	\cu(\Phi,X):=(\Phi \times X)/\sim.  
	\end{equation}
($\Phi$ has the discrete topology.)  Suppose $X$ is a mirrored CW complex and that we are given a cell $c$ of $X$ and a chamber $\gf\in \Phi$.  Then  $\gf\cdot c$ denotes the corresponding cell in $\cu(\Phi,X)$.  
Let $\cu^{(i)}$ denote the set of $i$-cells in $\cu(\Phi,X)$.  
Each such cell has the form $\gf\cdot c$ for some $\gf\in \Phi$ and $i$-cell $c$ of $X$.
\vspace{.2cm}

\noindent
\textbf{The classical realization}.  $\gD$ denotes the simplex of dimension $|S|-1$, with its codimension one faces indexed by $S$. In other words, the mirror $\gD_s$ is a codimension one face and $\gD_T$ (the intersection of the $\gD_s$ over all $s\in T$) is a face of codimension $|T|$.

The simplicial complex $\cu(W,\gD)$ is  the \emph{Coxeter complex} of $(W,S)$ while $\cu(\Phi, \gD)$ is the \emph{classical realization} of the building $\Phi$.

Tits constructed a representation of $W$ on $\rr^S$ called ``the contragredient of the canonical representation'' in \cite{bourbaki} and the ``geometric representation'' in \cite{dbook}.  The elements of $S$ are represented by reflections across the codimension one faces of a simplicial cone $C$.  The union  of translates of $C$ is denoted by $WC$.  It is a convex cone and $W$ acts properly on its interior $\ci$.  If $W$ is infinite, $WC$ is a proper cone. If $C^f$ denotes the complement of the nonspherical faces of $C$, then $\ci$ is equivariantly homeomorphic to $\cu(W,C^f)$.  Assume  $W$ is infinite.  Then the  image of $C-0$ in projective space can be identified with the simplex $\gD$ obtained by intersecting $C$ with some affine hyperplane; moreover,  $\gD^f$ is identified with the intersection of $C^f$ and this hyperplane.   The image of $\ci$ in projective space is then identified with $\cu(W,\gD^f)$.  Since this image is the interior of a topological disk (of dimension $|S|-1$), it follows that $\cu(W,\gD^f)$ is homeomorphic to a Euclidean space of that dimension.  In particular, it is contractible.
\vspace{.2cm}

\noindent
\textbf{Geometric realizations of posets}.  
Given a poset $\ct$, $\flag(\ct)$ denotes the set of finite chains in $\ct$, partially ordered by inclusion, i.e., an element of $\flag(\ct)$ is a finite, nonempty, totally ordered subset of $\ct$.  If $\ga=\{t_0,\dots,t_k\}\in \flag(\ct)$ where $t_0<\cdots<t_k$, then we will write $\ga:=\{t_0<\cdots<t_k\}$ and $\min \ga:=t_0$.  
$\flag(\ct)$ is an abstract simplicial complex with vertex set $\ct$ and with $k$-simplices the elements of $\flag(\ct)$ of cardinality $k+1$.  The corresponding topological simplicial complex is the \emph{geometric realization} of the poset $\ct$ and is denoted by $|\ct|$.  For example, if $\cp$ is the poset of proper subsets of $S$, partially ordered by inclusion, then its opposite poset, $\cp^{op}$, is the poset of nonempty faces of the simplex $\gD$. $\flag(\cp)$ is the poset of simplices in its barycentric subdivision, $b\gD$, and $|\cp|=b\gD$.
\vspace{.2cm}

\noindent
\textbf{The standard realization, $\boldsymbol{\cu(\Phi,K)}$}.  
As before, $\cs$ denotes the poset of spherical subsets of $S$. Put $K:=|\cs|$.  It is a subcomplex of $b\gD$ (provided $W$ is infinite).  The mirror structure on $\gD$ induces one on $K$.  More specifically, for each $s\in S$, put $K_s:=|\cs_{\ge \{s\}}|$ and for each $T\in \cs$, $K_T=|\cs_{\ge T}|$. ($K$ is the ``compact core'' of $\gD^f$.) $K$ is sometimes called the \emph{Davis chamber} of $(W,S)$ and $\cu(W,K)$, the \emph{Davis complex}.  Alternatively, $\cu(\Phi,K)$ is the geometric realization of the poset of spherical residues of $\Phi$ (see \cite{dbuild}).

By construction $\cu(\Phi, K)$ is locally finite (since $\Phi$ is assumed to have finite thickness).  It is proved in \cite{dbuild} that $\cu(\Phi,K)$ is contractible. 
\vspace{.2cm}

\noindent
\textbf{The realization $\boldsymbol{\cu(\Phi,\gD^f)}$}.  $\gD^f$ and $K$ have the same poset of faces (indexed by $\cs$) and there is a face-preserving deformation retraction $\gD^f\to K$.  That is to say, $\gD^f$ is a ``thickened version'' of $K$.  Similarly, $\cu(\Phi,\gD^f)$ is a thickened version of $\cu(\Phi,K)$.  Like $\cu(\Phi, K)$, the space $\cu(\Phi, \gD^f)$ has the advantage of being locally finite; however, the chamber $\gD^f$ is not compact whenever $\gD^f\neq \gD$.

\section{A $\boldsymbol{\cat(0)}$ metric on $\boldsymbol{\cu(\Phi,\gD^f)}$}\label{s:cat0}

Our goal in this section is to prove the following.

\begin{theorem}\label{t:cat}
Let $\Phi$ be a building of type $(W,S)$ with $W$ infinite.  Then there is a piecewise Euclidean, $\cat(0)$ metric on $\cu(\Phi, \gD^f)$.
\end{theorem}
\vspace{.2cm}

\noindent
\textbf{Review of the Moussong metric}.  
Suppose $T$ is a spherical subset of $S$.  
$W_T$ acts on $\rr^T$ via the canonical representation.  The \emph{Coxeter cell} of \emph{type $T$}, denoted  $P_T$, is defined to be the convex hull of the $W_T$-orbit of a point $x_0$ in the interior of the fundamental simplicial cone.  As examples, if $W_T$ is a product of $n$ copies of the cyclic group of order $2$, then $P_T$ is an $n$-cube; if $W_T$ is the symmetric group on $n+1$ letters, then $P_T$ is an $n$-dimensional permutohedron.  Its boundary complex, $\partial P_T$, is the dual of the Coxeter complex of $W_T$ (the Coxeter complex is a triangulation of the unit sphere in $\rr^T$).  The fact that $\partial P_T$ is dual to a simplicial complex means that  $P_T$ is a ``simple polytope''.  The isometry type of $P_T$ is determined once we choose the distance from $x_0$ to each of the bounding hyperplanes of the simplicial cone.  (We assume, without further comment, that such a choice of distance has been made for each $s\in S$.)  The intersection of $P_T$ with the fundamental simplicial cone is denoted $B_T$ and  called the \emph{Coxeter block} of type $T$.  It is a convex cell combinatorially isomorphic to a cube of dimension $|T|$ (because $P_T$ is simple).  One can identify $B_T$ with the subcomplex $|\cs_{\le T}|$
of $K$ in such a way that $x_0$ is identified with the vertex corresponding to $\emptyset$.  To be more precise, $B_T$ is the union of simplices of $\flag(\cs)$ whose maximum vertex is $\le T$, i.e., $\vert \cs_{\le T} \vert$ is a subdivision of $B_T$.

The convex polytope $B_T$ has two types of faces.  First, there are the faces which contain the vertex $x_0$.  Each such face is a Coxeter block of the form $B_V$ for some $V\subseteq T$.  The other type of face is the intersection of $B_T$ with the face of the fundamental simplicial cone fixed by $W_V$ for some nonempty $V\subseteq T$.  We denote such a face by $B_{T,V}$ and call it a \emph{reflecting face} of $B_T$.  For the purpose of unifying different cases, we shall sometimes write $B_{T,\emptyset}$ instead of $B_T$.

The \emph{Moussong metric} on $K$ is the piecewise Euclidean metric on $K$ in which each Coxeter block $B_T$ is given its Euclidean metric as a convex cell in $\rr^T$ (cf. \cite{moussong} or \cite[Section 12.1]{dbook}).  This induces a piecewise Euclidean metric on $\cu(W,K)$ as well as one on $\cu(\Phi,K)$.  The link of the central vertex corresponding to $\emptyset$ can be identified with a certain simplicial complex $L$ ($=L(W,S)$) called the \emph{nerve} of the Coxeter system.  The vertex set of $L$ is $S$ and a subset $T\subseteq S$ spans a simplex if and only if it is spherical.  Thus, the poset of simplices in $L$ (including the empty simplex) is $\cs$.  
The piecewise Euclidean metric on $K$ induces a piecewise spherical metric on $L$ such that whenever $m_{st}< \infty$, the length of the edge corresponding to $\{s,t\}$ is $\pi-\pi/m_{st}$.  Moussong proved that this piecewise spherical metric on $L$ is $\cat(1)$ and from this he deduced that the piecewise Euclidean metric on $\cu(W,K)$ is $\cat(0)$ (cf. \cite{moussong} and \cite[Section 12.3]{dbook}). Using this, it is proved in \cite{dbuild} that for any building $\Phi$, the Moussong metric on the standard realization,  $\cu(\Phi,K)$, is $\cat(0)$.
\vspace{.2cm}

\noindent
\textbf{A piecewise Euclidean metric on $\boldsymbol{\gD^f}$}.  We will define a cell structure on $\gD^f$ so that each cell will have the form $B_{T,V} \times [0,\infty)^m$ for some $T\in \cs$, $V\subseteq T$ and nonnegative integer $m$.  When $m>0$ such a cell will be noncompact.  There are two types of such cells.  First there are the compact cells  $B_{T,V}$ where $T\in \cs$ and $V\subseteq T$.  The remaining cells are in bijective correspondence with triples $(T,V,\ga)$ where $T\in \cs$, $V\subseteq T$ and $\ga\in \flag (\cp-\cs)$ is such that $T< \min \ga$.  Let $\cf$ be the set consisting of pairs $(T,V)$ and triples $(T,V,\ga)$,  where $T\in \cs$, $V\subseteq T$ and $\ga\in \flag (\cp-\cs)$ is such that $T< \min \ga$.  $\cf$ is partially ordered as follows:  
\begin{itemize}
\item
$(T', V')\le (T,V)$ if and only if $T'\subseteq T$ and $V'\subseteq V$,
\item
$(T',V',\ga')\le (T,V,\ga)$ if and only if $(T',V')\le (T,V)$ and $\ga'\subseteq \ga$, 
\item
$(T',V')\le (T,V,\ga)$ if and only if $(T',V')\le (T,V)$.
\end{itemize}
The cell $c(T,V,\ga)$ which corresponds to $(T,V,\ga)$ is defined to be $B_{T,V}\times [0,\infty)^\ga$, where $[0,\infty)^\ga$ means the set of all functions from the finite set $\ga$ to $[0,\infty)$.  For the most part, it will suffice to  deal with the case $V=\emptyset$ since the cells of the form $B_T\times[0,\infty)^\ga$ cover $\gD^f$. The piecewise Euclidean structure on $\gD^f$ will be defined by declaring each $B_T\times [0,\infty)^\ga$ to have the product metric. Thus, the piecewise Euclidean metric on $\gD^f$ will extend the one on $K$.

\begin{lemma}\label{l:cell}
$\gD^f$ has a decomposition into the cells, $\{B_{T,V}\}\cup \{ c(T,V,\ga)\}$, defined above, where $(T,V)$ and $(T,V,\ga)$ range over $\cf$.
\end{lemma}

To prove this, we need to set up a standard identification of the open cone of radius 1 on a $k$-simplex $\gs$ with the standard simplicial cone $[0,\infty)^{k+1}\subset \rr^{k+1}$.  Let $\{v_b\}_{b\in \cb}$ be the vertex set of $\gs$ for some finite index set $\cb$ and let $(x_b)_{b\in\cb}$ be barycentric coordinates on $\gs$.  Let $\rr^\cb$ be the Euclidean space of all functions $\cb\to \rr$.  Let $\sphere_+(\rr^\cb)$ denote the intersection of the standard simplicial cone $[0,\infty)^\cb$ with the unit sphere.
Thus, $\sphere_+(\rr^\cb)$ is an ``all right'' spherical simplex (i.e., all edge lengths and all dihedral angles are $\pi/2$).    Let $\{e_b\}_{b\in \cb}$ be the standard basis for $\rr^\cb$.  Define a homeomorphism $\theta_\cb:\gs\to \sphere_+(\rr^\cb)$, taking $v_b$ to $e_b$, by
\begin{equation}\label{e:theta}
\sum x_b v_b \to \sum x_be_b\  \bigg{/}\ (\sum x_b^2)^{1/2}.
\end{equation}
Let $r:[0,1)\to [0,\infty)$ be some fixed homeomorphism.  The open cone on $\gs$ can be regarded as the points $[t,x]$ in the  join, $v*\gs$, of $\gs$ with a point $v$ such that the join coordinate $t$ is $\neq 1$.  The homeomorphism from the open cone to $[0,\infty)^\cb$ is defined by $[t,x]\to r(t)\theta_\cb(x)$.  

\begin{proof}[Proof of Lemma~\ref{l:cell}]
Suppose $\ga\in \flag(\cp)$.  Then $\ga=\{T_0<\cdots <T_{k+l}\}$, where $T_k\in \cs$ and $T_{k+1}\notin \cs$.  Put $\ga':=\{T_0<\cdots <T_{k}\}$ and $\ga'':=\{T_{k+1}<\cdots <T_{k+l}\}$.  The simplex $\gs_{\ga'}$ lies in $B_{T_k}$ while the simplex $\gs_{\ga''}$ is in the nonspherical face $\gD_{T_{k+1}}$.  We have $\gs_\ga=\gs_{\ga'}*\gs_{\ga''}$ and a point in $\gs_\ga$ has coordinates $[t,x,y]$ where $t\in [0,1]$, $x\in \gs_{\ga'}$ and $y\in \gs_{\ga''}$.  The points in $\gs_{\ga}-\gs_{\ga''}$ are those where the join coordinate $t$ is $\neq 1$.  The identification $\gs_{\ga}-\gs_{\ga''} \to \gs_{\ga'} \times [0,\infty)^{\ga''}$ is given by $[t,x,y]\to (x,r(t)\theta_{\ga''}(y))$.
\end{proof}

Next we want to consider the link of the central vertex $v_\emptyset$ (corresponding to $\emptyset$) in this cell structure.  Let $\gD^{op}$ denote the simplex on $S$.  The nerve $L$ is a subcomplex of $\gdo$.  If $W$ is spherical, then $L=\gdo$, while if $W$ is infinite, then $L$ is a subcomplex of $\partial \gdo$.   Moreover, $\partial\gdo$ is a triangulation of the $(n-1)$-sphere, for $n=|S|$.  Assume $W$ is infinite. The link of $v_\emptyset$ in $\gD^f$ is a certain subdivision $L'$ of $\partial\gdo$, which we shall now describe.  The vertex set of $L'$ is the disjoint union $S\cup (\cp-\cs)$.  There are three types of simplices in $L'$:
\begin{enumerate1}
\item
simplices $\gs_T$ in $L$ corresponding to spherical subsets $T\in \cs_{>\emptyset}$,
\item
simplices $\gs_\ga$ in $b\partial\gD$ ($=b\partial\gdo$) corresponding to flags $\ga\in \flag (\cp-\cs)$,
\item
joins $\gs_T*\gs_\ga$, with $T\in \cs_{>\emptyset}$, $\ga\in \flag (\cp-\cs)$ and $T<\min \ga$.
\end{enumerate1}

\begin{lemma}\label{l:subdivision}
$L'$ is a subdivision of $\partial \gdo$ and $L\subseteq L'$ is a full subcomplex.
\end{lemma}

\begin{proof}
Suppose $U$ is a minimal element of $\cp-\cs$.  Let $\gs_U$ denote the corresponding simplex in $\partial\gdo$.  Introduce a ``barycenter'' $v_U\in \pgdo$ and then subdivide $\gs_U$ to a new simplicial complex $(\gs_U)'$ by coning off the simplices in $\partial \gs_U$.  Each new simplex will have the form $v_U*\gs_T$ for some $T\subset U$.

Next, let $U$ be an arbitrary element of $\cp-\cs$ and suppose by induction that we have defined the subdivision of  $(\gs_{U'})$ for each $U'\subset U$ and hence, a subdivision $(\partial \gs_U)'$ of $\partial \gs_U$.  Introduce a barycenter $v_U$ of $\gs_U$ and subdivide by coning off $(\partial \gs_U)'$.  Each new simplex will have the form $v_U*\gs_{U'}$ for some $U'\subset U$.  In other words, each new simplex will be either of type 2 or type 3 above.  It is clear that the subcomplex $L$ is full.
\end{proof}

The following lemma is also clear.

\begin{lemma}\label{l:link}
$L'$ is the link of the central vertex in the piecewise Euclidean cell structure on $\gD^f$. 
\end{lemma}

The piecewise Euclidean metric induces a piecewise spherical metric on $L'$ extending the given metric on $L$.  Since the link of the origin in $[0,\infty)^{k+1}$ is the all right spherical $k$-simplex we get the following description of the metric on $L'$.

\begin{lemma}\label{l:ps}
The  simplices in $L'$ have spherical metrics of the following types.
\begin{enumerate1}
\item
For $T\in \cs_{>\emptyset}$, the simplex $\gs_T$ in $L$ is the dual to the fundamental simplex for $W_T$ on the unit sphere in $\rr^T$.  In other words, for distinct elements $s$, $t$ in $T$, the edge corresponding to $\{s,t\}$ has length $\pi-\pi/m_{st}$.
\item
The simplex $\gs_\ga$ corresponding to $\ga\in \flag(\cp-\cs)$ has its all right structure.  In other words, each edge of $\gs_\ga$ has length $\pi/2$.
\item
The simplex $\gs_T*\gs_\ga$ has the structure of a spherical join.  In other words, the length of an edge connecting a vertex in $\gs_T$ to one in $\gs_\ga$ is $\pi/2$.
\end{enumerate1}
\end{lemma}

\begin{remark}\label{r:mfld}
The simplicial complex $L'$ is the nerve of a Coxeter group $(W',S')$ which contains $W$ as a special subgroup.  Namely, $S'$ is the disjoint union, $S \cup (\cp-\cs)$.  Two generators of $S$ are related as before.  If $U$, $V\in \cp-\cs$, then put $m(U,V):=2$ whenever $U\subset V$ or $V\subset U$ and $m(U,V):=\infty$, otherwise.  Similarly, if $s\in S$ and $U\in\cp-\cs$, then $m(s,U)=m(U,s)=2$ when $s\in U$ and it is $=\infty$, otherwise.  Since $L'$ is a triangulation of $S^{n-1}$, with $n=|S|$, this shows that any $n$ generator Coxeter group is a special subgroup of a Coxeter group which acts cocompactly on a contractible $n$-manifold (such a Coxeter group is said to be \emph{type} $HM^n$ in \cite{dbook}).
\end{remark}

\noindent
\textbf{$\boldsymbol{\cat(0)}$ and $\boldsymbol{\cat(1)}$ metrics}.  Gromov \cite{gromov} proved that a piecewise Euclidean metric on a polyhedron $Y$ is locally  $\cat(0)$ if and only if the link in $Y$ of each cell is $\cat(1)$.  In his proof that $L(W,S)$ was $\cat(1)$, Moussong \cite{moussong} gave a criteria for certain piecewise spherical structures on simplicial complexes to be $\cat(1)$.  We recall his criteria below.  A spherical simplex \emph{has size $\ge \pi/2$} if each of its edges has length $\ge \pi/2$.  A spherical simplex $\gs\subset \sphere^k\subset \rr^{k+1}$ with vertex set $\{v_0,\dots,v_{k+1}\}$ and edge lengths $l_{ij}:=\cos^\minus(v_i\cdot v_j)$ is determined up to isometry by the $(l_{ij})$.  Conversely, a symmetric $(k+1)\times (k+1)$ matrix $(l_{ij})$ of real numbers in [$0,\pi)$ can be realized as the set of edge lengths of a spherical simplex if and only if the matrix $(\cos l_{ij})$ is positive definite, cf. \cite[Lemma I.5.1, p.~513]{dbook}.  Suppose $N$ is a simplicial complex with a piecewise spherical structure (i.e., each simplex has the structure of a spherical simplex).  $N$ \emph{has size $\ge \pi/2$} if each of its simplices does.  $N$ is a \emph{metric flag complex} if it satisfies the following condition:  given any collection of vertices $\{v_0,\dots, v_k\}$ which are pairwise connected by edges, then $\{v_0,\dots, v_k\}$ is the vertex set of a simplex in $N$ if and only if the matrix of edge lengths $(l_{ij})$ can be realized as the matrix of edge lengths of an actual spherical simplex.  (In other words, if and only if  $(\cos l_{ij})$ is positive definite.)  Moussong's Lemma is the following.

\begin{lemma}\label{l:mlemma}
\textup{(Moussong's Lemma, \cite{moussong} or \cite[Appendix I.7]{dbook})}. 
Suppose a piecewise spherical simplicial complex $N$ has size $\ge \pi/2$.  Then $N$ is $\cat(1)$ if and only if it is a metric flag complex.
\end{lemma}

The next result follows immediately from our previous description of the piecewise spherical complex $L'$.

\begin{lemma}\label{l:L'}
$L'$ has size $\ge \pi/2$ and is a metric flag complex.
\end{lemma}

\begin{corollary}\label{c:L'}
$L'$ is $\cat(1)$.  Moreover, $L$ is a totally geodesic subcomplex.
\end{corollary}

\begin{proof}
The last sentence follows from the fact that $L$ is a full subcomplex.
\end{proof}

Since the link of any simplex in a metric flag complex of size $\ge\pi/2$ has the same properties (cf. \cite[Lemma 8.3]{moussong} or \cite[Lemma I.5.11]{dbook}), the link of any simplex in $L'$ is also $\cat(1)$.  Since the link in $\gD^f\cap \gD_V$ of any cell of the form $B_{T,V}\times [0,\infty)^\ga$ can be identified with the link of the corresponding simplex $\gs_T*\gs_\ga$ in $L'$, it follows that the link of each cell in $\gD^f$ is $\cat(1)$.


\begin{corollary}\label{c:cat0}
The piecewise Euclidean metric on $\cu(W,\gD^f)$ is $\cat(0)$.
\end{corollary}

\begin{proof}
The union of $W_T$-translates of  a Coxeter block $B_T$ in $\cu(W,\gD^f)$  is a Coxeter cell $P_T$ and the complete inverse image of $B_T$ in $\cu(W,\gD^f)$ is a disjoint union of copies of $P_T$.  Hence, $\cu(W,\gD^f)$ has a cell structure in which the cells are either translates of Coxeter cells of the form $P_T$ or translates of cells of the form $P_T\times [0,\infty)^\ga$ for some $\ga\in \flag((\cp-\cs)_{>T})$.  In either case the link of such a cell in $\cu(W,\gD^f)$ is identified with the link of the corresponding cell in $\gD^f$.  By Corollary~\ref{c:L'}, $\cu(W,\gD^f)$ is locally $\cat(0)$. 
A space is $\cat(0)$ if and only if it is locally $\cat(0)$ and simply connected (cf. \cite[p.~119]{gromov} or \cite[Ch.~II.4]{bh}).  $\cu(W,\gD^f)$ is contractible (hence, simply connected), since it is homotopy equivalent to $\cu(W,K)$.  Therefore, it is $\cat(0)$.
\end{proof}

We also want to consider links of cells of the form $B_{T,V}$ or $B_{T,V} \times [0,\infty)^\ga$ in various cell complexes.  (Here $B_{T,V}$ is a reflecting face of $B_T$.)

\begin{lemma}\label{l:links}
Suppose $c$ is a cell in $\gD^f$ of the form $c=B_{T,V}$ or $c=B_{T,V}\times [0,\infty)^\ga$ for some $V\subseteq T\in \cs_{>\emptyset}$ and $\ga\in\flag ((\cp-\cs)_{>T})$. Let $d=B_T$ or $B_T\times [0,\infty)^\ga$ be the corresponding larger cell in $\gD^f$ and let $\gs_d$ be the corresponding simplex in $L'$.  Then
\begin{align}
\Lk (c,\gD^f)&=\Lk(\gs_d,L')*\gt(V)\label{e:L1}\\
\Lk (c, \cu(W,\gD^f)&=\Lk (\gs_d,L')* \sphere^V\label{e:L2}
\end{align}
Here $\sphere^V$ is the unit sphere in the canonical representation of $W_V$ on $\rr^V$ and $\gt(V)\subset \sphere^V$ is the fundamental simplex. 

Moreover, suppose that $\Phi$ is a building of type $(W,S)$, that $R$ is the spherical residue of type $V$ containing the base chamber and that $\sphere(R):=\cu(R,\gt(V))$ is the spherical realization of $R$.  Then
\begin{equation}\label{e:L3}
\Lk (c, \cu(\Phi,\gD^f)=\Lk (\gs_d,L')* \sphere(R).
\end{equation}
\end{lemma}

\begin{proof}
Let $\gD^f_V$ denote the face of $\gD^f$ fixed by $W_V$ (cf. \eqref{e:T}. Then
\begin{align*}
\Lk(c,\gD^f)&= \Lk(c,\gD^f_V)* \Lk(\gD^f_V,\gD^f)\\
&=\Lk (d,\gD^f)* \gt(V)\\
&=\Lk(\gs_d,L')*\gt(V)
\end{align*}
and similarly for formulas \eqref{e:L2} and \eqref{e:L3}.
\end{proof}
We can now prove the main result of this section.

\begin{proof}[Proof of Theorem~\ref{t:cat}]
Any spherical building, such as $\sphere(R)$, is $\cat(1)$ (e.g., see \cite{dbuild}) and the spherical  join of two $\cat(1)$-spaces is $\cat(1)$ (e.g., see \cite{bh}). So, the theorem follows from \eqref{e:L3}.  Alternatively, it can  be proved from Corollary~\ref{c:cat0} by using the argument in \cite[\S 11]{dbuild}.
\end{proof}

\vspace{.2cm}

\noindent
\textbf{A variation}.  In Section~\ref{s:morse} we will need the following modification of the previous construction.  Given a subset $U\subseteq S$, we will define a new piecewise Euclidean metric on $\gD^f-\gD^U$ and then show that it induces $\cat(0)$ metrics on $\cu(W_{S-U}, \gD^f-\gD^U)$ and $\cu(R, \gD^f-\gD^U)$ for any $(S-U)$-residue $R$ of $\Phi$.

For each spherical subset $T$, let $C^*_T\subset \rr^T$ be the simplicial cone determined by the bounding hyperplanes of $B_T$ passing through the vertex $x_0$.  (In other words, $C^*_T$ is the dual cone to the fundamental simplicial cone.)  Let $\cs_{S-U}:=\cs(W_{S-U}, S-U)$ be the poset of spherical subsets of $S-U$ and let $L_{S-U}:=L(W_{S-U},S-U)$ be the nerve of $W_{S-U}$.  $\gD^f-\gD^U$ has a cell structure with cells of the following two types:
\begin{enumeratea}
\item
$B_{T,V}$, where $T\in \cs_{S-U}$ and $V\subseteq T$,
\item
$B_T\times [0,\infty)^\ga$, where $\ga\in \cp-\cs_{S-U}$ and $T,V$ are as above.
\end{enumeratea}
As before, each such cell is given the natural product metric.  
In  effect  we are putting $\gU(\gD) \cup \gD^U$ at infinity.

The piecewise Euclidean metric on $\gD^f-\gD^U$ induces one $\cu(R,\gD^f-\gD^U)$.  Let us describe the link, $L'_{S-U}$, of the central vertex in the new metric on $\gD^f-\gD^U$.  The vertex set of  $L'_{S-U}$ is $(S-U)\cup (\cp-\cs_{S-U})$.  As before, there are three types of simplices:
\begin{enumerate1}
\item
simplices $\gs_T$ in $L_{S-U}$ corresponding to spherical subsets $T\in \cs_{S-U}$,
\item
simplices $\gs_\ga$ in $b\partial\gD$ ($=b\partial\gdo$) corresponding to flags $\ga\in \flag (\cp-\cs_{S-U})$,
\item
joins $\gs_T*\gs_\ga$, with $T\in \cs_{S-U}$, $\ga\in \flag (\cp-\cs_{S-U})$ and $T<\min \ga$.
\end{enumerate1}
and $L'_{S-U}$ can be identified with a subdivision of $\pgdo$.  Moreover, just as before, $L'_{S-U}$ has size $\ge\pi/2$ and is a metric flag complex; hence, it  
is $\cat(1)$. This proves the following.

\begin{theorem}\label{t:cat2}
Suppose $U\subseteq S$ and $R$ is any $(S-U)$-residue in $\Phi$.  Then the piecewise Euclidean metric on $\cu(R,\gD^f-\gD^U)$, defined above, is $\cat(0)$.
\end{theorem}

\section{Metric spheres in $\boldsymbol{\cu(\Phi, \gD^f)}$}\label{s:morse}

An $n$-dimensional  cell complex $X$ is $CM$ (for ``Cohen-Macaulay'') if $\wt{H}^*(X)$ is concentrated in degree $n$ and is a free abelian group in that degree.  Similarly, an $n$-dimensional, noncompact, contractible space $X$ is $SI$ (for ``spherical at infinity'') if $H^*_c(X)$ is concentrated in degree $n$ and is a free abelian group in that degree. We want to prove that $\cu(\Phi, \gD^f)$ is $SI$.  This is a consequence of Theorems \ref{t:bbm} and \ref{t:pi2} below.

Suppose $N$ is a $\cat(1)$, piecewise spherical polyhedron and that $p\in N$.  Let $B(p,\pi/2)\subseteq N$ denote the open ball of radius $\pi/2$ centered at $p$.  Define a space $PN_p$, called $N$ \emph{punctured at} $p$, by $PN_p:=N-B(p,\pi/2)$.  We are interested in this concept when $N=\Lk(c)$, the link of a cell $c$ in some $\cat(0)$ complex $X$.  In this case we will write $\plk_p(c)$ for $PN_p$ and call it the \emph{punctured link} of $c$ at $p$.

\begin{theorem}\label{t:bbm}
\textup{(Brady, McCammond and Meier \cite{bbm})}.
 Let $X$ be a $\cat(0)$, piecewise Euclidean cell complex (with finitely many shapes of cells). If for each
 cell $c$ in $X$
and for each $p\in \Lk(c)$, the spaces 
$\Lk(c)$ and $\plk_p (c)$ are
$(n-\dim c)$-acyclic, then $X$ is $n$-acyclic at infinity.  In particular, if $X$ satisfies this condition and  is $n$-dimensional,  then it is $SI$.
\end{theorem}

\begin{remarks*}
In \cite{bbm} the hypothesis of the above theorem is that $X$ is the universal cover of a finite, nonpositively curved complex; however, the proof clearly works with a weaker hypothesis such as finitely many shapes of cells (which holds in our case).  The proof of the theorem uses Morse theory for polyhedral complexes.  Roughly, it goes as follows.  Let $\gr:X\to \rr$ be the distance from some base point $x_0$, i.e., $\gr(x):=d(x,x_0)$. The spheres $S(r)$ of radius $r$ centered at $x_0\in X$ are the level sets of $\gr$.  Call a point $x$ a \emph{critical point} (of $\gr$) if it is the closest point  to $x_0$ in some closed cell $c$.  It follows from the $\cat(0)$ hypothesis that the critical points are isolated.  If, for sufficiently small $\geps$, there is no critical point in the annular region between $S(r+\geps)$ and $S(r)$, then $S(r+\geps)$ and $S(r)$ are homeomorphic.  On the other hand, the effect of crossing a critical point $x\in S(r)$ is to remove a contractible neighborhood of $x$ in $S(r)$ and replace it by the punctured link $\plk(x)_p$ where $p$ is the direction at $x$ of the geodesic from $x$ to $x_0$.  (If $x$ lies in the relative interior of a $k$-dimensional cell $c$, then $\Lk(x)\cong \sphere^{k-1} * \Lk(c)$.)  So, the effect  on the homotopy type of the level sets is to replace a contractible neighborhood by a copy of a suspension of a  punctured link.  It follows that, under the hypotheses of Theorem~\ref{t:bbm}, each metric sphere is $CM$.  Since $X$ is $\cat(0)$, metric balls are contractible and since
$H^*_c(X)=\varinjlim H^*(B(r),S(r))$ is concentrated in the top degree, 
$X$ is $SI$.
\end{remarks*}

\begin{theorem}\label{t:pi2}
\textup{(Dymara--Osajda \cite{dymo} and Schulz \cite{someone})}.
Suppose $R$ is a spherical building of type $(W_T,T)$ and that $\sphere(R)$ ($:=\cu(R, \gt(T))$) is its spherical realization.  Then for any $p\in \sphere(R)$, the space $P\sphere (R)_p$ is $CM$.
\end{theorem}

An immediate corollary to Theorems~\ref{t:bbm} and \ref{t:pi2} is the following.

\begin{theorem}\label{t:SI}
\textup{(cf.~\cite{bs}).}  With notation as in Section~\ref{s:cat0} and above, 
given any building $\Phi$, its realization 
$\cu(\Phi,\gD^f)$ is $SI$.
\end{theorem}

As a corollary to Theorem~\ref{t:cat2}
we get the following relative version of Theorem~\ref{t:bbm}.

\begin{theorem}\label{t:SI2}
Suppose $U\subseteq S$ and $R$ is any $(S-U)$-residue in $\Phi$. 
Then $\cu(R,\gD^f-\gD^U)$ is $SI$.
\end{theorem}

\begin{corollary}\label{c:SI3}
Suppose $U\subseteq S$.  Then $\cu(\Phi,\gD^f-\gD^U)$ is $SI$.
\end{corollary}

\begin{proof}
$\cu(\Phi,\gD^f-\gD^U)$ is the disjoint union of the spaces $\cu(R,\gD^f-\gD^U)$ where $R$ ranges over the $(S-U)$-residues.  By the previous theorem, the compactly supported cohomology of each such component is free abelian and concentrated in the top degree.
\end{proof}

\section{Cohomology with finite support and compact support}\label{s:support}
Given a CW complex $Y$, $C^*_{\fin}(Y)$ denotes the complex of finitely supported cellular cochains on $Y$ and $H^*_\fin(Y)$ its cohomology.  When $Y$ is only required to be a topological space, $H^*_c(Y)$ denotes its compactly supported singular cohomology, i.e.,
\[
H^*_c(Y):=\varinjlim H^*(Y, Y-C),
\]
where the direct limit is over all compact subsets $C\subseteq Y$.  If $Y$ is a  locally finite CW complex,  $H^*_\fin(Y)\cong H^*_c(Y)$.

Suppose $Z$, $Z'$ are mirrored spaces over $S$.  A map $F:Z\to Z'$ is \emph{mirrored} if $F(Z_T)\subseteq Z'_T$ for all $T\subseteq S$; it is a \emph{mirrored homotopy equivalence} if $F\vert_{Z_{T}}:Z_T\to Z'_T$ is a homotopy equivalence for all $T$ (including $T=\emptyset$).

Given a mirrored space $X$, let us say that $\gU(X)$ is \emph{collared in $X$} if there is an increasing family $\cn=\{N_\geps\}_{\geps\in (0,a]}$ of open neighborhoods of $\gU(X)$ (``increasing means that $N_\geps \subseteq N_{\geps'}$ whenever $\geps <\geps'$) such that the following two properties hold:
\begin{enumeratei}
\item
\(
\bigcap_{\geps \in (0,a]} N_\geps = \gU(X)
\) and
\item
For each $\geps>0$, the inclusion $\gU(X)\hookrightarrow \ol{N}_\geps$ is a mirrored homotopy equivalence.
\end{enumeratei}

For example, if $X$ is the simplex $\gD$, then $\gU(\gD)$ is collared in $\gD$.  (Proof: $\gU(\gD)$ is a union of closed faces and we can take $N_\geps$ to be its $\geps$-neighborhood in $\gD$.)  More generally, if $X$ is any finite CW complex, then $\gU(X)$ is collared in $X$.

Property (ii) implies that $\cu(\Phi, \gU(X))\hookrightarrow \cu(\Phi, N_\geps)$ is a homotopy equivalence.
To simplify notation, in what follows we often write $\cu_X$ instead of $\cu(\Phi,X)$.

\begin{lemma}\label{l:cfin}
Suppose $X$ is a finite, mirrored CW complex .   Then 
\[
H^*_c(\cu_{X^f})=H^*_\fin(\cu_X,\cu_{\gU(X)}).
\]
\end{lemma}

\begin{proof}
Since $X$ is a finite complex, 
\(
H^*_\fin (\cu_X,\cu_{\gU(X)})=H^*_c(\cu_X,\cu_{\gU(X)}).
\)  
Since $\gU(X)$ is collared in $X$, we have a family $\cn=\{N_\geps\}$ of open neighborhoods. For a given $N= N_\geps\in \cn$, let $\partial N$ denote the boundary of $\ol{N}$.  Since $N$ is mirrored homotopy equivalent to $\gU(X)$,
\[
H^*_\fin(\cu_X,\cu_{\gU(X)})\cong H^*_c(\cu_X,\cu_{\ol{N}})\cong 
H^*_c(\cu_{(X-N)},\cu_{\partial N}),
\]
where the second isomorphism is an excision. As $\geps\to 0$, $X-N_\geps \to X^f$, so any compact subset $C$ of $\cu_{X^f}$ lies within some $\cu_{X-N_\geps}$.  Hence,
\[
H^*_c(\cu_{X^f})=\varinjlim H^*(\cu_{X^f},\cu_{X^f}-C)=H^*_c(\cu_{X-N},\cu_{\partial N}).
\]
Combining these equations, we get the result.
\end{proof}

\begin{example}
$\gD^f$ is  a thickened version of $K$ and $\cu(\Phi,\gD^f)$ is a thickened version of  $\cu(\Phi,K)$.  Hence, $\cu(\Phi,\gD^f)$ and $\cu(\Phi,K)$ are homotopy equivalent.  On the other hand, their compactly supported cohomology groups can be completely different.  For example, suppose $W$ is the free product of three copies of $\zz/2$.  Then $\gD^f$ is a triangle with its vertices deleted while $K$ is a tripod.  $\cu(W,\gD^f)$ can be identified with the hyperbolic plane while $\cu(W,K)$ is the regular trivalent tree.  (This is the familiar picture of the congruence $2$ subgroup of $PSL(2,\zz)$ acting on the hyperbolic plane.)  The compactly supported cohomology of $\cu(W,\gD^f)$ is that of the plane (i.e., it is concentrated in degree $2$ and is isomorphic to $\zz$ in that degree), while the compactly supported cohomology of $\cu(W,K)$ is that of a tree (i.e., it is concentrated in degree $1$ and is a countably generated, free abelian group in that degree).
\end{example}

\section{Cohomology with coefficients in $\ci(A)$}\label{s:A}
Let $A$ ($=A(\Phi)$) be the free abelian group of finitely supported, $\zz$-valued functions on $\Phi$.  For each subset $T$ of $S$, define a $\zz$-submodule of $A$: 
\[
A^T:=\{f\in A\mid f \text{ is constant on each residue of type $T$}\}.
\]
Note that $A^\emptyset =A$.  Also note that $A^T=0$ whenever $T$ is  not spherical.

We have $A^U\subset A^T$ when $T\subset U$.  Let $A^\ut$ be the $\zz$-submodule of $A^T$ spanned by the $A^U$ with $T$ a proper subset of $U$.  Put 
\[
D^T:=A^T/A^\ut.
\]

\begin{remark}
When $\Phi=W$, $A$ is the group ring $\zz W$ and $A^T$ consists of elements in $\zz W$ which are constant on each left coset $wW_T$.  Let us assume $T$ is spherical (otherwise $A^T=0$).  Let $\In (w):=\{s\in S\mid l(ws)<l(w)\}$ be the set of letters with which a reduced expression for $w$ can end. Since $T$ is spherical, each left coset of $W_T$ has a unique representative $w$ which has a reduced expression ending in the longest element of $W_T$; hence, this representative has $T\subseteq \In(w)$.  As a basis for $D^T$, we can take  images of the elements in $A^T$ corresponding to cosets which have longest representatives with $T=\In(w)$.
\end{remark}

\begin{definition}\label{d:cochains}
The abelian group $A$ and the family of subgroups $\{A^T\}_{T\subseteq S}$ define a ``coefficient system'' $\ci(A)$ on $X$ so that the \emph{$i$-cochains with coefficients in $\ci(A)$} are given by
\[
\cac^i(X; \ci(A)):= \prod_{c\in X^{(i)}} A^{S(c)}
\]
where $X^{(i)}$ denotes the set of $i$-cells in $X$.
\end{definition}

We continue to write $\cu_X$ for $\cu(\Phi, X)$.  Let $\cu^{(i)}$ denote the set of $i$-cells in $\cu_X$.  
Given a chamber $\gf\in \Phi$ and an $i$-cell $c$ in $X$, let $\gf\cdot c$ denote the corresponding $i$-cell in $\cu_X$.  Given a finitely supported function $\ga:\cu^{(i)}\to \zz$, and an $i$-cell $c\in X^{(i)}$ with $S(c)$ spherical, we get an element $f\in A^{S(c)}$ defined by
\(
f(\gf):=\ga(\gf\cdot c).
\)
Of course, when $S(c)$ is not spherical, $A^{S(c)}=0$.
So, for any finite, mirrored CW complex $X$, 
this establishes an isomorphism 
\begin{equation}\label{e:cochainiso}
\cac^i(X;\ci(A))\cong C^i_\fin (\cu_X, \cu_{\gU(X)}).  
\end{equation}
In other words, the isomorphism \eqref{e:cochainiso} is given by identifying a finitely supported function on the inverse image in $\cu^{(i)}$ of a cell $c\in X^{(i)}$ with a function on $\Phi$ (i.e., with an element of  $A$) which is constant on  $S(c)$-residues containing the cells $\gf\cdot c$.     

The coboundary maps in $\cac^*(X;\ci(A))$ are defined by using these isomorphisms to transport the coboundary maps on finitely supported cochains to $\cac^*(X;\ci(A))$.  This means that the coboundary maps in $\cac^*(X;\ci(A))$ are defined by combining the usual coboundary maps in $C^*(X)$ with the inclusions $A^U\hookrightarrow A^T$ for $U\supset T$. So, we have the following.

\begin{lemma}\label{l:finite} For $X$  a finite, mirrored CW complex, 
\[
\ch^*(X;\ci(A))=H^*_\fin (\cu_X,\cu_{\gU(X)})=H^*_c(\cu_{X^f}).
\]
\end{lemma}

\begin{proof}
The second equation is from Lemma~\ref{l:cfin}.
\end{proof}

\begin{remark} 
If $\Phi =W$, then $A=\zz W$.  $\cac^*(X;\ci(\zz W))$ can be interpreted as the equivariant cochains on $\cu(W,X)$ with coefficients in $\zz W$.  When $X$ is a finite complex, Lemma~\ref{l:finite} asserts
\[
\cac^*(X;\ci(\zz W )=C^*_\fin(\cu(W,X),\cu(W,\gU(X))).
\]
The corresponding cohomology groups were studied in \cite{ddjo2, d98}.  In particular, when $X=K$, these cohomology groups are isomorphic to $H^*(W;\zz W)$.  
\end{remark}

In what follows  the coefficient system is usually $\ci(A)$ and we shall generally omit it from our notation, writing $\ch(\ )$ and $\cac(\ )$ instead of $\ch(\ ;\ci(A))$ and $\cac(\ ;\ci(A))$.  (In other words, the coefficients $\ci(A)$ are implicit when we use the calligraphic $\cac$ or $\ch$ notation.) 
As usual,  $\gD$ is the simplex of dimension $n=|S|-1$ with its codimension one faces indexed by $S$. 

\begin{theorem}\label{t:gd}
\textup{(cf.~\cite[Theorem B]{dj} as well as \cite{ddjo, dym}).}  
$\ch^*(\gD)$ is concentrated in degree $n$ and is free abelian; moreover, $\ch^n(\gD)=D^\emptyset$.  More generally, for any subset $U$ of $S$, $\ch^*(\gD, \gD^U)$ is concentrated in degree $n$ and is free abelian and 
\[
\ch^n(\gD,\gD^U)=A\big/\sum_{s\in S-U} A^s
\]
\end{theorem}

\begin{proof}
By Lemma~\ref{l:finite}, $\ch^*(\gD)=H^*_c(\cu_{\gD^f})$ and by Theorem~\ref{t:SI}, the right hand side is concentrated in degree $n$ and is free abelian.  The cochain complex looks like
\[
\cdots \to \cac^{n-1}(\gD) \to \cac^n(\gD)\to 0,
\]
where $\cac^n(\gD)=A$ and $\cac^{n-1}(\gD)=\bigoplus A^s$.   It follows that cohomology in degree $n$ is the quotient
\[
A\big/\sum_{s\in S} A^s = D^\emptyset.
\]
Similarly, $\ch^*(\gD,\gD^U)=H^*_c(\cu_{\gD^f-\gD^U})$, and by Corollary~\ref{c:SI3} the right hand side is concentrated in degree $n$ and is free abelian.  Since $\cac^n(\gD,\gD^U)=A$ and $\cac^{n-1}(\gD, \gD^U)=\bigoplus_{s\notin U} A^s$, we get the final formula in the theorem.
\end{proof}

\begin{corollary}\label{c:Dfree}
$D^\emptyset$ is free abelian.
\end{corollary}

\begin{remark}\label{r:duality}
We saw in Section~\ref{s:geom} that whenever $W$ is infinite, $\cu(W, \gD^f)$ is homeomorphic to Euclidean space $\rr^n$.  Hence, the compactly supported cohomology of $\cu(W, \gD^f)$ is that of $H^*_c(\rr^n)$.  Similarly, $\cu(R,\gD^f-(\gD^f)^U)$ is homeomorphic to $\rr^n$, for each $(S-U)$-residue $R$.
\end{remark}

\begin{remark}\label{r:pd}
$\cu(W,\gD^f)$  is a thickened version of $\cu(W,K)$. In 
the proof of Corollary~\ref{c:cat0} we explained the cellulation of $\cu(W,K)$ by ``Coxeter cells.''  The corresponding cellular chain complex, $C_*(\cu(W,K))$,  has the form
\begin{equation*}
\zz W\ \longleftarrow \bigoplus_{s\in S}H^s\ \longleftarrow
\bigoplus_{T\in \cs^{(2)}}H^T \longleftarrow\cdots
\end{equation*}
where $H^T$ is the representation induced from the sign representation of $W_T$ and where $\cs^{(k)}$ is the set of spherical subsets with $k$ elements (cf.~\cite[\S 8]{ddjo}).  The cochain complex $\cac^*(\gD)$ for $n-2\le *\le n$, looks like
\begin{equation*}
\cdots\mapright{}\ \bigoplus_{T\in \cs^{(2)}}A^T\mapright{}\  \bigoplus_{s\in S}A^s\mapright{}\ \zz W .
\end{equation*}
So, $C_*(\cu(W,K))$ and $C^{n-*}_c(\cu(W,\gD^f))$ are Poincar\'e dual. 
\end{remark}

\section{The Decomposition Theorem}\label{s:decomp}
We  want to prove a version of Theorem~\ref{t:gd} for lower dimensional spherical faces of $\gD$. Suppose $T$ is a spherical subset of $S$. Put $\gs=\gD_T$ and $m=n-|T|$.  Let $U$ be an arbitrary subset of $S-T$.  We continue the policy of omitting the coefficient system $\ci(A)$ from our notation.

\begin{proposition}\label{p:sigma}
Each of the following cohomology groups is concentrated in the top degree and is a free abelian group in that degree:  
\[
\ch^*(\gs,\gs^U),\  \ch^*(\gs^U,\partial (\gs^U)), \text{ and }\ \ch^*(\gs^U)
\]
(The top degrees are $m$, $m-1$, and $m-1$, respectively.)  Moreover,
\begin{align}
\ch^m(\gs,\gs^U)&=A^T\big/\sum_{s\in (S-T)-U} A^{T\cup\{s\}},\label{e:1}  \\
\ch^{m-1}(\gs^U,\partial (\gs^U))&=\sum_{s\in U} A^{T\cup\{s\}},\label{e:2}\\
\ch^{m-1}(\gs^U)&=\sum_{s\in U}A^{T\cup\{s\}} \big/ \sum_{\substack{s\in U\\t\in (S-T)-U}} A^{T\cup
\{s,t\}},\label{e:3}
\end{align}
\end{proposition}

\begin{proof}
We first prove $\ch^*(\gs,\gs^{U})$ is concentrated in degree $m$ and is free abelian.
The proof  is by induction on the number of elements in $T$.  It holds for $|T|=0$ by Theorem~\ref{t:gd}.  Suppose $T=T'\cup \{s\}$, $U'=U\cup \{s\}$ and $\gt=\gD_{T'}$.   The exact sequence of the triple $(\gt,\gt^{U'},\gt^{U})$ gives
\begin{equation*}
\cdots \to\ch^{*-1}(\gs,\gs^U) \to \ch^*(\gt, \gt^{U'})\to \ch^*(\gt,\gt^{U}) \to\cdots
\end{equation*}
(This uses the excision, $\ch^{*-1}(\gt^{U'},\gt^{U})\cong \ch^{*-1}(\gs,\gs^U)$.)
By inductive hypothesis, the last two terms are free abelian and concentrated in degree $m+1$.  Hence, $\ch^*(\gs,\gs^{U})$ is concentrated in degree $m$.  It is free abelian since it injects into a free abelian group.

That the other cohomology groups are free abelian and are concentrated in the top degree follows from various exact sequences.  For example, for $\ch^*(\gs^U,\partial (\gs^U))$, consider
the sequence of the triple $(\gs,\partial \gs, \gs^{S-U})$:
\begin{equation}\label{e:triple1}
\to\ch^{*-1}(\gs^U,\partial(\gs^U))\to \ch^*(\gs,\partial\gs)\to \ch^*(\gs,\gs^{S-U})\to 
\end{equation}
where we have used the excision $\ch^{*-1}(\partial\gs, \gs^{S-U})\cong \ch^{*-1}(\gs^U,\partial(\gs^U))$ to identify the first term.  The second and third terms are free abelian and concentrated in degree $m$; hence, the first term is free abelian and concentrated in degree $m-1$.
For $\ch^*(\gs^U)$, we have the exact sequence of the pair $(\gs,\gs^U)$:
\begin{equation}\label{e:pair1}
\to \ch^{*-1}(\gs^U)\to \ch^*(\gs,\gs^U)\to \ch^*(\gs)\to 
\end{equation}
where the second and third terms are free abelian and concentrated in degree $m$.  Hence, $\ch^*(\gs^U)$ is free abelian and is concentrated in degree $m-1$.

It remains to verify formulas \eqref{e:1}, \eqref{e:2} and \eqref{e:3}. We have $\cac^m(\gs,\gs^U)=A^T$ and $\cac^{m-1}(\gs,\gs^U)=\bigoplus_{s\notin U} A^{T\cup\{s\}}$, so
\[
\ch^m(\gs,\gs^U)=A^T \big{/} \sum_{s\notin U} A^{T\cup\{s\}},
\]
proving \eqref{e:1}.  
(In particular, $\ch^m(\gs)=D^T$ and $\ch^n(\gD)=D^\emptyset$.)
In the exact sequence \eqref{e:triple1}, we have $\ch^m(\gs,\partial \gs)=A^T$ and, by \eqref{e:1}, $\ch^m(\gs,\gs^{S-U})=A^T/\sum_{s\in U} A^{T\cup\{s\}}$; hence, \eqref{e:2}.
Using \eqref{e:1} to calculate the second and third terms of \eqref{e:pair1}, we get 
\[
\ch^{m-1}(\gs^U)=\sum_{s\in S-T}A^{T\cup\{s\}} \big/ \sum_{s\in (S-T)-U} A^{T\cup\{s\}}
\]
and this can be rewritten as  \eqref{e:3}.
\end{proof}

By Proposition~\ref{p:sigma}, 
$D^T$ is the free abelian group $\ch^m(\gD_T)$.  So,  for each $T\in \cs$, we can choose a splitting $\gi_T:D^T\to A^T$ of the projection map $A^T\to D^T$.  

\begin{definition}\label{d:ha}
Let $\ha^T:=\gi_T(D^T)$.  (It is a direct summand of the free abelian group $A^T$.)
\end{definition}

\begin{proposition}\label{p:sums}
\begin{align}
\ch^m(\gs,\gs^U)&=\bigoplus_{\substack{V\supseteq T\\V-T\subseteq U}} \ha^V \label{e:sum1}\\
\ch^{m-1}(\gs^U, \partial(\gs^U))&=\bigoplus_{\substack{V\supset T\\(V-T)\cap U\neq\emptyset}} \ha^V\label{e:sum2}\\
\ch^{m-1}(\gs^U)&=\bigoplus_{\substack{V\supset T\\V-T\subseteq U}} \ha^V \label{e:sum3}
\end{align}
\end{proposition}

\begin{proof}
Assume by induction that \eqref{e:sum1} through \eqref{e:sum3} hold for $\dim \gs=m-1$ and assume as well that they  hold  when $\dim \gs=m$ and  $U$ is replaced by $U'$ with $|U'|<|U|$.  Write $U=\{s\}\cup U'$, for some $s\in U$.  Consider the exact sequence of the triple $(\gs^U,\gs_s\cup \partial (\gs^U),\partial (\gs^U))$:  
\begin{equation}\label{e:triple2}
0\to \ch^{m-1}(\gs^{U'},\partial(\gs^{U'}))\to \ch^{m-1}(\gs^U,\partial (\gs^U))\to \ch^{m-1}(\gs_s, (\gs_s)^{S-U'})\to 0,
\end{equation}
where we have used the excisions $\ch^{*}(\gs^U,\gs_s\cup \partial (\gs^U))=\ch^{*}(\gs^{U'},\partial(\gs^{U'}))$ and $\ch^{*}(\gs_s\cup \partial (\gs^U),\partial (\gs^U))=\ch^{*}(\gs_s,(\gs_s)^{S-U'})$ to rewrite the first and third terms.
By induction, 
\begin{align*}
\ch^{m-1}(\gs^{U'},\partial(\gs^{U'}))
&=\bigoplus_{\substack{V\supset T\\V\cap U'\neq\emptyset}} \ha^V   ,\\
\ch^{m-1}(\gs_s,(\gs_s)^{S-U'}) &=
\bigoplus_{\substack{V\supseteq T\cup\{s\}\\V\subseteq T\cup\{s\}\cup U'}} \ha^V 
\end{align*}
Substituting these into the last two terms of \eqref{e:triple2}, we get
\[
\ch^{m-1}(\gs^U,\partial(\gs^U))=\bigoplus_{\substack{V\supset T\\V\cap U\neq \emptyset}} \ha^V ,
\]
which is \eqref{e:sum2}.  
Next consider the Mayer-Vietoris sequence of $\gs^U=\gs^{U'}\cup \gs_s$:
\begin{equation}\label{e:MV1}
0\to \ch^{m-2}((\gs_s)^{U'})\to \ch^{m-1}(\gs^U)\to \ch^{m-1}(\gs^{U'})\oplus \ch^{m-1}(\gs_s)\to 0.
\end{equation}
By induction,
\[
\ch^{m-2}((\gs_s)^{U'})=\bigoplus_{\substack{V\supset T\cup\{s\}\\V-(T\cup\{s\})\subseteq U'}} \ha^V
\quad\text{and}\quad
\ch^{m-1}(\gs^{U'})=\bigoplus_{\substack{V\supset T\\V-T\subseteq U'}} \ha^V
\]
and $\ch^{m-1}(\gs_s))=\ha^{T\cup\{s\}}$.  Substituting these into \eqref{e:MV1} we get
\[
\ch^{m-1}(\gs^U)=\bigoplus_{\substack{V\supset T\\V-T\subseteq U}} \ha^V,
\]
which is \eqref{e:sum3}.  Sequence \eqref{e:pair1} is
\[
0\to \ch^{m-1}(\gs^U)\to \ch^m(\gs, \gs^U)\to \ch^m(\gs)\to 0.
\]
Substituting \eqref{e:sum3} for the first term and $\ha^T$ for the third, we get formula \eqref{e:sum1} for the middle term.
\end{proof}

We have $\ch^m(\gs,\partial \gs)=A^T$.  Hence, in the special case $U=S-T$ formulas \eqref{e:sum1} and \eqref{e:sum3} give the following theorem (cf. \cite{s}, \cite[Thm. 9.11]{ddjo}, \cite[Cor. 3.3]{ddjo2}).

\begin{theorem}\label{t:decomp}
\textup{(The Decomposition Theorem).}  For each subset $T$ of $S$
\[
A^T=\bigoplus_{V\supseteq T} \ha^V.
\]
\end{theorem}

For any $\zz$-submodule $B\subset A$, put $B^T:=A^T\cap B$.  
Suppose we have a direct sum decomposition of $\zz$-modules, $A=B\oplus C$, so that for each $T\subseteq S$, $A^T=B^T\oplus C^T$.  
As explained in \cite[\S 2]{ddjo2} this leads to a decomposition of coefficient systems $\ci(A)=\ci(B)\oplus \ci(C)$ so that for any mirrored CW complex $X$ we have a decomposition of cochain complexes $\cac^*(X;\ci(A))=\cac^*(X;\ci(B))\oplus \cac^*(X;\ci(C))$.  Since
\[
(\ha^V)^T=
\begin{cases}
\ha^V	&\text{if $V\supseteq T$,}\\
0		&\text{otherwise,}
\end{cases}
\]
the formula in the Decomposition Theorem satisfies $(A^\emptyset)^T=\bigoplus_{V\supseteq T}(\ha^V)^T$, for all $T\subseteq S$.  So, we get a decomposition of coefficient systems $\ci(A)=\bigoplus \ha^V$ and a corresponding decomposition of cochain complexes:
\begin{equation}\label{e:decomp}
\cac^*(X;\ci(A))=\bigoplus_{V} \cac^*(X;\ci(\ha^V))
\end{equation}

\section{Cohomology of buildings}\label{s:cohomology}
Just as in \cite[Thm. 3.5]{ddjo2}, the Decomposition Theorem (Theorem~\ref{t:decomp}) implies the following.  

\begin{theorem}\label{t:main}
\textup{(cf. \cite[Cor.~8.2]{dj}, \cite[Thm.~10.3]{ddjo}, \cite[Thm.~3.5]{ddjo2}).}  Suppose $X$ is a finite, mirrored CW complex.  Then
\[
H^*_c(\cu(\Phi,X^f))\cong \bigoplus_{T\in \cs} H^*(X,X^{S-T})\otimes \hA^T.
\]
\end{theorem}

\begin{proof}
By Lemma~\ref{l:finite}, $H^*_c(\cu(\Phi,X^f))$ is the cohomology of $\cac^*(X;\ci(A))$.  
Formula \eqref{e:decomp} gives a decomposition of cochain complexes:
\begin{equation*}
\cac^*(X;\ci(A))=\bigoplus_{T\in \cs} \cac^*(X;\ci(\ha^T))
\end{equation*}
We have
\[
\cac^k(X;\ci(\ha^T))=\prod_{c\in X^{(k)}}\ha^T\cap A^{S(c)}
=\prod_{\substack{c\in X^{(k)}\\ c\not\subseteq X^{S-T}}} \ha^T.
\]
So, an element of $\cac^k(X;\ci(\ha^T))$ is just an ordinary $\ha^T$-valued cochain on $X$ which vanishes on $X^{S-T}$, i.e., 
\[
\cac^*(X;\ci(\ha^T))=C^*(X,X^{S-T})\otimes \ha^T; \quad \text{so,}
\]
\[
\cac^*(X;\ci(A))=\bigoplus_{T\in \cs} C^*(X,X^{S-T})\otimes \ha^T.
\]
Taking cohomology, we get the result.
\end{proof}

The most important special case of the previous theorem is the following.

\begin{corollary}\label{c:main}
\[
H^*_c(\cu(\Phi, K))\cong \bigoplus_{T\in \cs} H^*(K,K^{S-T})\otimes \ha^T.
\]
\end{corollary}

\section{The $G$-module structure on cohomology}
Assume $X$ has a $W$-finite mirror structure (i.e., $X=X^f$).
Suppose $G$ is a group of automorphisms of $\Phi$.  Then $A$, $A^T$, $A^\ut$ and $D^T$ are naturally right $G$-modules and so is the cochain complex $C^*_c(\cu(\Phi,X))$ as is its cohomology.  The discussion in Section~\ref{s:cohomology}  is well adapted to studying the $G$-module structure of $H^*_c(\cu(\Phi,X))$.  As in \cite{ddjo2}, we should not expect a direct sum splitting of $G$-modules analogous nonequivariant splitting of Theorem~\ref{t:main}; rather there should be a filtration of cohomology by $G$-submodules with associated graded terms similar to those in the direct sum.  We show below that this is the case.

For each nonnegative integer $p$ define a $G$-submodule $F_p$ of $A$ by 
\[
F_p:=\sum_{|T|\le p} A^T.
\] 
This gives a decreasing filtration:
\begin{equation}\label{e:f}
A=F_0\supset\cdots \supset F_p\supset F_{p+1}\cdots .
\end{equation}
As in \cite{ddjo2}, it follows from the Decomposition Theorem that the associated graded terms are
\[
F_p/F_{p+1}=\bigoplus_{|T|=p} D^T.
\]

As in Section~\ref{s:A}, we get a coefficient system $\ci(F_p)$, a cochain complex 
$\cac^*(X;\ci(F_p))$ and corresponding cohomology groups $\ch^*(X;\ci(F_p))$.
The filtration \eqref{e:f} leads to a filtration  of $\ch^*(X;\ci(A))$ ($=H^*_c(\cu(\Phi,X))$) by right $G$-modules,

\begin{equation}\label{e:f2}
\ch^*(X;\ci(A))\supset \cdots \ch^*(X;\ci(F_p))\supset \ch^*(X;\ci(F_{p+1}))\cdots .
\end{equation}
The Decomposition Theorem implies that
\[
0\mapright{}\,\ch^*(X;\ci(F_{p+1}))\mapright{}\,\ch^*(X;\ci(F_p))\mapright{}\,
H^*\left(\frac{\cac^*(X;\ci(F_p))}{\cac^*(X;\ci(F_{p+1}))}\right)\mapright{} \,0
\]
is short exact.  From this we deduce the following.

\begin{theorem}\label{t:Gmodule}
Suppose $G$ is a group of automorphisms of $\Phi$.  Then the filtration \eqref{e:f2} of $H^*_c(\cu(\Phi,X))$ by right $G$-modules  has associated graded term,
\[
\ch^*(X;F_p)/\ch^*(X;F_{p+1})\cong \bigoplus _{T\in \cs^{(p)}} H^*(X,X^{S-T})\otimes D^T.
\]
\end{theorem}
\vspace{.2cm}

\noindent
\textbf{Cocompact $\boldsymbol{G}$-action}.  In this final paragraph we assume that there are only finitely many $G$-orbits on $\Phi$ and that $X$ is a finite complex. 
These hypotheses imply that the quotient space $\cu(\Phi,X)/G$ is compact and since $X$ is a finite complex, each of  the cohomology groups $H^*(X)$ is finitely generated. As a corollary to Theorem~\ref{t:Gmodule}, we have the following.

\begin{corollary}\label{c:fingen}
With the above hypotheses,  $H^*_c(\cu(\Phi,X))$ is a finitely generated $G$-module.  
\end{corollary}

\begin{proof}
The assumption that $G$ has only finitely many orbits of chambers implies that each $A^T$ is finitely generated $G$-module; hence, so is each $D^T$. 
\end{proof}

\begin{example}
Suppose $G$ is chamber transitive, i.e., suppose $\Phi$ consists of a single $G$-orbit.  Choose a base chamber $\gf_0$, let $B$ be its stabilizer and for each $T\in \cs$, let $G_T$ be the stabilzer of the $T$-residue containing $\gf_0$.  Then $A$ is the $G$-module $\zz (G/B)$ of finitely supported functions on $G/B$ and $A^T$ can be identified with $\zz(G/G_T)$.  For each $U\supset T$, we have a natural inclusion $\zz(G/G_U)\to \zz(G/G_T)$  induced by the projection $G/G_T\to G/G_U$ and $D^T$ is the $G$-module formed by dividing out the sum of the images of  $\zz(G/G_U)$ in $\zz(G/G_T)$ for all $U\supset T$.
\end{example}

Here are some further corollaries to Corollary~\ref{c:main} and Theorem~\ref{t:Gmodule}.

\begin{corollary}\label{c:dim}
Suppose $\gG\subseteq G$ is a cocompact lattice in $G$.  Then $H^*(\gG;\zz \gG)$ has, in filtration degree $p$, associated graded $\gG$-module
\[
\bigoplus _{T\in \cs^{(p)}} H^*(K,K^{S-T})\otimes D^T.
\]
In particular, when $\gG$ is torsion-free, its cohomological dimension is given by 
\[
\cohd (\gG)=\max\{k\mid H^k(K,K^{S-T})\neq 0, \text{for some $T\in \cs$}\}.
\]
\end{corollary}

A torsion-free group $\gG$ is an \emph{$n$-dimensional duality group} if $H^*(\gG;\zz\gG)$ is free abelian and concentrated in dimension $n$.
Following \cite[Definition 6.1]{dm},  we say that the nerve of a Coxeter system has \emph{punctured homology concentrated in dimension $n$} if for all $T\in \cs$, $\wt{H}^*(K^{S-T})$ is free abelian and concentrated in dimension $n$. Corollary~\ref{c:main}  gives us a (correct) proof of the following result, stated in \cite[Theorem 6.3]{dm}.

\begin{corollary}\label{c:duality}
Suppose $\gG\subseteq G$ is a torsion-free, cocompact lattice in $G$. Then the following are equivalent.
\begin{enumerate1}
\item
$\gG$ is an $n$-dimensional duality group.
\item
$W$ is an $n$-dimensional virtual duality group.
\item
The nerve of $(W,S)$ has punctured homology concentrated in dimension $n-1$.
\end{enumerate1}
\end{corollary}

\end{document}